\documentclass[1p,final]{elsarticle}
\usepackage{amsfonts,color,morefloats,pslatex}
\usepackage{amssymb,amsthm, amsmath,latexsym}

\newtheorem{theorem}{Theorem}
\newtheorem{lemma}[theorem]{Lemma}

\newtheorem{open}[theorem]{Open Problem}

\newcommand{\tr}{{\mathrm{Tr}}}

\newcommand{\gf}{{\mathrm{GF}}}
\newcommand{\PG}{{\mathrm{PG}}}

\newcommand{\wt}{{\mathtt{wt}}}

\newcommand{\C}{{\mathsf{C}}} 
\newcommand{\cV}{{\mathcal{V}}}

\newcommand{\cT}{{\mathcal{T}}}

\newcommand{\bc}{{\mathbf{c}}}

\newcommand{\bv}{{\mathbf{v}}}

\begin{document}

\begin{frontmatter}



\title{A family of ovoids in $\PG(3, 2^m)$ from cyclic codes}


\author{Cunsheng Ding}
\address{Department of Computer Science and Engineering, 
The Hong Kong University of Science and Technology,
Clear Water Bay, Kowloon, Hong Kong, China}

\begin{abstract}
Ovoids in $\PG(3, q)$ have been an interesting topic in coding theory, combinatorics, 
and finite geometry for a long time. So far only two families are known. The first 
is the elliptic quadratics and the second is the Tits ovoids. In this article, we 
present a family of ovoids in $\PG(3, 2^m)$ for all $m$ which are from a family of 
irreducible cyclic codes.     
\end{abstract}


\end{frontmatter}

\section{Introduction} 

A cap in $\PG(3, \gf(q))$ is a set of points in $\PG(3, \gf(q))$ such that 
no three are collinear. Let $q>2$. For any cap $\cV$ in $\PG(3, \gf(q))$, 
we have $|\cV| \leq q^2+1$ (see \cite{Bose47}, \cite{Seiden50} and \cite{Qvist52} 
for details).   

In the projective space $\PG(3, \gf(q))$ with $q>2$, an \emph{ovoid}\index{ovoid} $\cV$ 
is a set of $q^2+1$ points such that no three of them are collinear (i.e., on the same line). In other words, an ovoid is a $(q^2+1)$-cap (a cap with $q^2+1$ points) in 
$\PG(3, \gf(q))$, and thus a maximum cap.

A \emph{classical ovoid} $\cV$ can be defined as the set of all points given by 
\begin{eqnarray}
\cV=\{(0,0,1, 0)\} \cup \{(x,\, y,\, x^2+xy +ay^2,\, 1): x,\, y \in \gf(q)\}, 
\end{eqnarray} 
where $a \in \gf(q)$ is such that the polynomial $x^2+x+a$ has no root in $\gf(q)$. 
Such ovoid is called an \emph{elliptic quadric}\index{quadric}, as the points 
come from a non-degenerate elliptic quadratic form.  

For $q=2^{2e+1}$ with $e \geq 1$, there is an ovoid which is not an elliptic quadric, 
and is called the \emph{Tits oviod}\index{Tits ovoid} \cite{Tits60}. It is defined by 
\begin{eqnarray}
\cT=\{(0,0,1,0)\}\cup \{(x,\,y,\, x^{\sigma} + xy +y^{\sigma+2},\,1): x, \, y \in \gf(q)\}, 
\end{eqnarray}   
where $\sigma=2^{e+1}$. 

For odd $q$, any ovoid is an elliptic quadric (see \cite{Barlotti} and \cite{Panella}).
For even $q$, Tits ovoids are the only known ones which are not elliptic quadratics. 
In the case that $q$ is even, 
the elliptic quadrics and the Tits ovoid are not equivalent \cite{Willems}. 
For further information about ovoids, the reader is referred to \cite{OKeefe} and  
\cite{HS98}. The objective of this article is to present a family of ovoids in 
$\PG(3, 2^m)$ from irreducible cyclic codes over $\gf(2^m)$.  

\section{Ovoids in $\PG(3, \gf(q))$ and $[q^2+1, 4, q^2-q]$ Codes}\label{sec-ovoidcodes} 

Let $\cV$ be an ovoid in $\PG(3, \gf(q))$ with $q>2$. Denote by 
$$ 
\cV=\{\bv_1, \bv_2, \cdots, \bv_{q^{2}+1}\} 
$$ 
where each $\bv_i$ is a column vector in $\gf(q)^4$. Let $\C_{\cV}$ be the linear code over $\gf(q)$ 
with generator matrix 
\begin{eqnarray}
G_{\cV}=\left[\bv_1 \bv_2 \cdots \bv_{q^{2}+1}\right].  
\end{eqnarray} 
Note that $\cV$ intersects each plane in either one point or $q+1$ points. It then 
follows that $\C_{\cV}$ has only the nonzero 
weights $q^2-q$ and $q^2$. The code is clearly projective. Solving the first two 
Pless power moments, one obtains the following weight enumerator of the code:   
\begin{eqnarray}\label{eqn-wtdistovidcode}
1+(q^2-q)(q^2+1)z^{q^2-q}+(q-1)(q^2+1)z^{q^2}. 
\end{eqnarray} 
Hence, $\C_{\cV}$ is a $[q^2+1, 4, q^2-q]$ code over $\gf(q)$. 
Its dual is a $[q^2+1, q^2-3, 4]$ code.

Linear codes over $\gf(q)$ with parameters $[q^2+1, 4, q^2-q]$ are special and attractive 
due to the following result \cite[p. 192]{Bier}. 

\begin{lemma}\label{lem-ovoidcodewtlem} 
Any linear code over $\gf(q)$ with parameters $[q^2+1, 4, q^2-q]$ must have the weight enumerator 
of (\ref{eqn-wtdistovidcode}). 
\end{lemma} 

The following theorem follows from Lemma \ref{lem-ovoidcodewtlem}. 

\begin{theorem}\label{thm-ovoidcodespel}
The dual of any linear code over $\gf(q)$ with parameters $[q^2+1, 4, q^2-q]$ must have parameters 
$[q^2+1, q^2-3, 4]$.  
\end{theorem} 

The following conclusion then follows from Theorem \ref{thm-ovoidcodespel}. 

\begin{theorem}\label{thm-codeovoid}
Let $\C$ be a $[q^2+1, 4, q^2-q]$ code over $\gf(q)$. Let $G$ be a generator matrix 
of $\C$. Then the column vectors of $G$ form an ovoid in $\PG(3, \gf(q))$.   
\end{theorem}

Due to Theorems \ref{thm-codeovoid} and   
\ref{thm-ovoidcodespel}, linear codes over $\gf(q)$ with parameters $[q^2+1, 4, q^2-q]$ 
are called \emph{ovoid codes}\index{ovoid code}. Another special feature of ovoid codes 
is that they meet the Griesmer bound. 
Linear codes over $\gf(q)$ with parameters $[q^2+1, q^2-3, 4]$ are almost-MDS codes. Note that 
the weight distribution of general almost-MDS codes is not known, though that of MDS codes 
is determined.

\section{A Family of Cyclic Codes with Parameters $[q^2+1, 4, q^2-q]$} 

The objective of this section is to present a family of irreducible cyclic 
codes over $\gf(q)$ with parameters $[q^2+1, 4, q^2-q]$, which give a family 
of ovoids in $\PG(3, \gf(q))$ by Theorem \ref{thm-codeovoid}.     

Let $r$ be a power of $q$ and $q$ be a power of a prime $p$. 
Let $N>1$ be an integer dividing $r-1$, and put $n=(r-1)/N$.
Let $\alpha$ be a primitive element of $\gf(r)$ and let $\theta=\alpha^N$.
The set
\begin{eqnarray}\label{eqn-irrecode2}
\C(r,N) =
 \{(\tr_{r/q}(\beta), \tr_{r/q}(\beta \theta), ...,
              \tr_{r/q}(\beta \theta^{n-1})) :
    \beta \in \gf(r)\}
\end{eqnarray}
is called an {\em irreducible cyclic $[n, m_0]$ code} over $\gf(q)$,
where $\tr_{r/q}$ is the trace function from $\gf(r)$ onto $\gf(q)$,
$m_0$ is the multiplicative order of $q$ modulo $n$ and
$m_0$ divides $m$.

Let $\zeta_p=e^{2\pi \sqrt{-1}/p}$, and $\chi(x)=\zeta_p^{\tr_{r/p}(x)}$, where $\tr_{r/p}$ is
the trace function from $\gf(r)$ to $\gf(p)$. Then $\chi$ is an
additive character of $\gf(r)$. 
Let $\alpha$ be a fixed primitive element of $\gf(r)$.
Define $C_{i}^{(N,r)}=\alpha^i \langle \alpha^{N} \rangle$ for $i=0,1,...,N-1$, where
$\langle \alpha^{N} \rangle$ denotes the
subgroup of $\gf(r)^*$ generated by $\alpha^{N}$. The cosets $C_{i}^{(N,r)}$ are
called the {\em cyclotomic classes} of order $N$ in $\gf(r)$. 
The {\em Gaussian periods} are defined by
$$
\eta_i^{(N,r)} =\sum_{x \in C_i^{(N,r)}} \chi(x), \quad i=0,1,..., N-1,
$$
where $\chi$ is the canonical additive character of $\gf(r)$.  

To determine the weight distribution of some irreducible cyclic codes later, we need the
following lemma \cite{DY13}.

\begin{lemma}\label{lem-haw}
Let $e_1$ be a positive divisor of $r-1$
and let $i$ be any integer with $0 \le i <e_1$.
We have the following multiset equality:
\begin{eqnarray*}\label{eqn-fgfg0}
\begin{array}{rl}
\left\{\left\{xy: y \in \gf(q)^*, \ x \in C_i^{(e_1,r)}\right\}\right\}  
= \frac{(q-1)\gcd((r-1)/(q-1), e_1)}{e_1}*C_i^{(\gcd((r-1)/(q-1),e_1),r)},
\end{array} \nonumber \\ 
\end{eqnarray*}
where $\frac{(q-1)\gcd((r-1)/(q-1), e_1)}{e_1}*C_i^{(\gcd((r-1)/(q-1),e_1),r)}$ denotes the multiset
in which each element in
the set $C_i^{(\gcd((r-1)/(q-1),e_1),r)}$ appears in the multiset with multiplicity exactly $\frac{(q-1)\gcd((r-1)/(q-1), e_1)}{e_1}$.
\end{lemma}

\begin{proof}
We just prove the conclusion for $i=0$. The proof is similar for $i\neq0$ since
$$C_i^{(\gcd((r-1)/(q-1),e_1),r)}=\alpha^i C_0^{(\gcd((r-1)/(q-1),e_1),r)}.$$
Note that every $y \in \gf(q)^*$ can be expressed as $y=\alpha^{\frac{r-1}{q-1}\ell}$
for an unique $\ell$ with $0 \le \ell < q-1$ and every $x \in C_0^{(e_1,r)}$ can be expressed as
$x=\alpha^{e_1j}$ for an unique $j$ with $0 \le j < (r-1)/e_1$. Then
we have
$$
xy=\alpha^{\frac{r-1}{q-1}\ell + e_1j}.
$$
It follows that
$$
xy=\alpha^{\frac{r-1}{q-1}\ell + e_1j}=
(\alpha^{\gcd((r-1)/(q-1),e_1)})^{ \frac{r-1}{(q-1)\gcd((r-1)/(q-1),e_1)}\ell + \frac{e_1}{\gcd((r-1)/(q-1),e_1)}j}.
$$
Note that
$$
\gcd\left( \frac{r-1}{(q-1)\gcd((r-1)/(q-1),e_1)}, \frac{e_1}{\gcd((r-1)/(q-1),e_1)}\right)=1.
$$
When $\ell$ ranges over $0 \le \ell < q-1$ and $j$ ranges over $0 \le j < (r-1)/e_1$,
$xy$ takes on the value $1$ exactly $\frac{q-1}{e_1}\gcd((r-1)/(q-1), e_1)$ times.

Let $x_{i_1} \in C_0^{(e_1,r)}$ for $i_1=1$ and $i_1=2$, and let $y_{i_2} \in \gf(q)^*$ for $i_2=1$ and $i_2=2$.
Then $\frac{x_1}{x_2} \in C_0^{(e_1,r)}$ and $\frac{y_1}{y_2} \in \gf(q)^*$. Note that
$x_1y_1=x_2y_2$ if and only if $\frac{x_1}{x_2} \frac{y_1}{y_2} =1$. Then the
conclusion of the lemma for the case $i=0$ follows from the discussions above.
\end{proof}

Let $N>1$ be an integer dividing $r-1$, and put $n=(r-1)/N$.
Let $\alpha$ be a primitive element of $\gf(r)$ and let $\theta=\alpha^N$.
Let $Z(r,a)$ denote the number of solutions $x \in \gf(r)$ of the equation $\tr_{r/q}(ax^{N})=0$.
We have then by Lemma \ref{lem-haw}
\begin{eqnarray*}\label{eqn-forall1}
Z(r,a) 
&=& \frac{1}{q} \sum_{y \in \gf(q)} \sum_{x \in \gf(r)} \zeta_p^{\tr_{q/p}(y \tr_{r/q} (ax^{N}))} \nonumber \\
&=& \frac{1}{q} \sum_{y \in \gf(q)} \sum_{x \in \gf(r)} \chi(yax^{N}) \nonumber \\
& =& \frac{1}{q} \left[ q + r-1 + N \sum_{y \in \gf(q)^*} \sum_{x \in C_{0}^{(N,r)}} \chi(ya x) \right] \nonumber \\
& =& \frac{1}{q} \left[ q + r-1 + (q-1)\gcd(\frac{r-1}{q-1}, N)\cdot \sum_{z \in C_{0}^{\left(\gcd\left(\frac{r-1}{q-1},N\right),r\right)}} \chi(az) \right]. 
\end{eqnarray*}

Then the Hamming weight of the codeword 
\begin{eqnarray}\label{eqn-mycword}
\bc(\beta)=(\tr_{r/q}(\beta), \tr_{r/q}(\beta \theta), ...,
              \tr_{r/q}(\beta \theta^{n-1}))
\end{eqnarray} 
in the irreducible cyclic code of (\ref{eqn-irrecode2}) is equal to
\begin{eqnarray}\label{eqn-wtmain}
n-\frac{Z(r,\beta)-1}{N}=
\frac{(q-1)\left(r-1-\gcd\left(\frac{r-1}{q-1}, N\right) \eta_{k}^{ \left(\gcd\left(\frac{r-1}{q-1},N\right),r\right)} \right)}{qN}.
\end{eqnarray}

Below we present a family of two-weight cyclic codes which are in fact ovoid codes. 

\begin{theorem}\label{thm-iccode}
Let $q=2^s$, where $s \geq 2$. Let $m=4$ and $N=q^2-1$. Then the code  $\C(r, N)$ over $\gf(q)$ of \eqref{eqn-irrecode2} 
has parameters $[q^2+1, 4, q^2-q]$ and weight enumerator 
\begin{eqnarray}
1+(q^2-q)(q^2+1)z^{q^2-q} + (q-1)(q^2+1)z^{q^2}.  
\end{eqnarray}
\end{theorem}

\begin{proof}
Let 
$$ 
N_1=\gcd((r-1)/(q-1), N)=q+1. 
$$
It then follows from \cite{BMW82} that 
$$ 
\eta_0^{(N_1, r)}=-(q^2-q+1), \ \eta_i^{(N_1, r)}=q-1 \mbox{ for } 1 \leq i \leq q. 
$$ 
Let $\beta \in C_i^{(N_1, r)}$, where $0 \leq i \leq q$. By \eqref{eqn-wtmain}, the Hamming weight of the codeword 
$\bc(\beta)$ in \eqref{eqn-mycword} is given by 
\begin{eqnarray*}
\wt(\bc(\beta)) 
&=& \frac{(q-1)\left(r-1-\gcd\left(\frac{r-1}{q-1}, N\right) \eta_{k}^{ \left(\gcd\left(\frac{r-1}{q-1},N\right),r\right)} \right)}{qN} \\ 
&=& \frac{(q-1)\left(r-1-(q+1) \eta_{k}^{ \left(N_1,r\right)} \right)}{qN} \\ 
&=& \left\{ 
\begin{array}{ll}
q^2 & \mbox{ if } i=0, \\
q^2-q & \mbox{ if } 1 \leq i \leq q. 
\end{array}
\right. 
\end{eqnarray*} 
It is obvious that the dual code of  $\C(r, N)$ has minimum distance at least $2$. Let $w_1=q^2-q$ and $w_2=q^2$. 
Let $A_{w_1}$ and $A_{w_2}$ denote the number of codewords with weight $w_1$ and $w_2$ in $\C(r, N)$, respectively. 
The first two Pless power moments then become 
\begin{eqnarray*}
A_{w_1}+A_{w_2}=q^4-1 \mbox{ and } w_1A_{w_1}+w_2A_{w_2}=q^3(q-1)(q^2+1). 
\end{eqnarray*} 
Solving this set of equations above yields 
$$ 
A_{w_1}=(q^2-q)(q^2+1), \ A_{w_2}=(q-1)(q^2+1). 
$$ 
This completes the proof. 
\end{proof} 

The following problem is open. 

\begin{open} 
Are the ovoids given by $\C(r,N)$ equivalent to the elliptic quadratics or the Tits ovoids?  
\end{open}

\section{Concluding remarks} 

Let $\C$ be a $[q^2+1, 4, q^2-q]$ code over $\gf(q)$, i.e., an ovoid code. 
The weight distribution of $\C^\perp$ is given below. 

\begin{theorem}\label{thm-iccodedual}
Let $q \geq 4$, and let $\C$ be a $[q^2+1, 4, q^2-q]$ code over $\gf(q)$. 
Then the weight distribution of $\C^\perp$ is given by 
\begin{eqnarray}
q^4 A^\perp_\ell &=& \binom{q^2+1}{\ell}(q-1)^\ell + 
     u \sum_{i+j=\ell} \binom{q^2-q}{i}(-1)^i \binom{q+1}{j} (q-1)^j  + \nonumber \\
   && v\left[ (-1)^\ell \binom{q^2}{\ell} +(-1)^{\ell-1} (q-1) \binom{q^2}{\ell-1} \right]
\end{eqnarray}
for all $4 \leq \ell \leq q^2$, and 
$$ 
q^4 A^\perp_{q^2+1}=(q-1)^{q^2+1}+u(q-1)^{q+1}+v(q-1), 
$$ 
where 
\begin{eqnarray}
u=(q^2-q)(q^2+1), \ v=(q-1)(q^2+1) 
\end{eqnarray} 
and $A^\perp_\ell$ denotes the number of codewords of weight $\ell$ in $\C^\perp$. 
\end{theorem}

Any ovoid code $\C$ over $\gf(q)$ holds $3$-designs which are documented below.  

\begin{theorem}\label{thm-steinerplane}
Let $q \geq 4$ and let $\C$ be a $[q^2+1, 4, q^2-q]$ code over $\gf(q)$. Then the supports of the codewords of weight 
$q^2-q$ in $\C$ form a design with parameters 
$$ 
3\mbox{-}(q^2+1, \ q^2-q, \ (q-2)(q^2-q-1)). 
$$ 
The complement of this design is a $3$-$(q^2+1, q+1, 1)$ Steiner system (i.e., 
an inversive plane).  

Furthermore, the supports of all the codewords of weight 
$4$ in $\C^\perp$ form a $3$-$(q^2+1, 4, q-2)$ design. 
\end{theorem} 

Another open problem is if the $3$-designs held in the codes $\C(r, N)$ are equivalent to 
those held in the ovoid code from the elliptic quadric or the Tits ovoid.

\end{document}